\documentclass[11pt]{amsart}

\usepackage{amsmath,amssymb,amsthm,mathrsfs,mathtools}
\usepackage{hyperref}
\usepackage{enumitem}
\usepackage[margin=1.15in]{geometry}

\hypersetup{colorlinks=true,linkcolor=blue,citecolor=blue,urlcolor=blue}

\newtheorem{theorem}{Theorem}[section]
\newtheorem{proposition}[theorem]{Proposition}
\newtheorem{lemma}[theorem]{Lemma}
\newtheorem{corollary}[theorem]{Corollary}
\theoremstyle{definition}

\theoremstyle{remark}

\newcommand{\ZZ}{\mathbb Z}
\newcommand{\QQ}{\mathbb Q}
\newcommand{\Li}{\operatorname{Li}}

\newcommand{\ch}{\operatorname{ch}}
\newcommand{\td}{\operatorname{td}}
\newcommand{\rk}{\operatorname{rk}}
\newcommand{\Pic}{\operatorname{Pic}}

\title[All Genus Gopakumar-Vafa Reconstruction]{All Genus Gopakumar-Vafa Reconstruction from Central Difference Forcing}
\author{Xiaobin Li}
\address{School of Mathematics, Southwest Jiaotong University, Chengdu, Sichuan 611756, China}
\date{September 18, 2026}

\begin{document}
\begin{abstract}
Generating functions in curve counting theory mix two different kinds of information: multiple covers mix curve degrees, while the topological string coupling mixes BPS genera. We show that a single central difference operator separates these two effects in a precise and invertible way for the all genus Gopakumar-Vafa expansion along an effective ray of a Calabi-Yau threefold.

Writing
\[
y=\left(2\sin\frac{\lambda}{2}\right)^2,
\]
we combine the Bryan-Pandharipande cover polynomials with the corresponding Fej\'er factors and obtain a Chebyshev-Fej\'er ghost transform for the normalized difference forcing. The transform admits an explicit operator-valued M\"obius inverse. Its $y$-adic filtration is exactly the BPS genus filtration: the forcing modulo $y^{G+1}$ is equivalent to the complete collection of BPS invariants of genera at most $G$. Thus the genus zero reconstruction from the classical specialization is the first associated graded piece of an all genus hierarchy.

We also characterize the integral image of the transform. Integral BPS forcing is equivalent to a two stage condition: Witt-Dwork type Frobenius congruences for the underlying ghost coordinates, together with an additional Fej\'er divisibility condition imposed by the central difference operator. In prime power degree this gives a strengthened Dwork congruence containing the full Fej\'er factor.

For local curves, the formalism has concrete geometric realizations. A super-rigid elliptic curve gives an exact Dedekind $\eta$ expression for the forcing. In the balanced higher genus local curve theory, the upper degree two BPS invariants are realized as top Chern numbers on natural stacks of square roots over symmetric products. These results connect central difference equations with BPS reconstruction, arithmetic ghost structures, and the geometry of local curves.
\end{abstract}

\maketitle

\section{Introduction and main results}
\label{sec:intro}

Let $X$ be a Calabi-Yau threefold and fix a primitive effective curve class $\beta$. We work along the ray $\ZZ_{>0}\beta$ and write
\[
Q=q^\beta.
\]
For $d\ge1$ and $g\ge0$, let $n_{g,d}$ denote the finite BPS coefficients entering the raywise Gopakumar-Vafa transform. Throughout the paper, the structural statements are conditional only on the existence of such a finite BPS expansion in each degree; no independent foundational construction of the all genus GV invariants is required.

The raywise BPS-resummed contribution is the usual Gopakumar-Vafa multiple cover expansion \cite{GopakumarVafaI,GopakumarVafaII}. This structure is part of the broader threefold curve counting framework relating Gromov-Witten, Donaldson-Thomas and stable-pair theories \cite{MNOPI,MNOPII,PandharipandeThomas2009,PandharipandeThomas2010}; modern geometric constructions of GV invariants via vanishing cycles provide a further rigorous BPS framework \cite{MaulikToda2018}. We use only the formal finite BPS expansion on the chosen ray:
\begin{equation}
F^{\mathrm{BPS}}(\lambda,Q)
=
\sum_{d\ge1}\sum_{g\ge0}n_{g,d}
\sum_{k\ge1}\frac{Q^{dk}}{k}
\left(2\sin\frac{k\lambda}{2}\right)^{2g-2}.
\label{eq:GV-intro}
\end{equation}
For the resolved conifold, Alim proved a central difference equation for the full nonconstant potential and observed analogous degreewise equations for the genus zero GV-resummed contribution of a general Calabi-Yau threefold \cite{Alim2023}. A complementary recent resurgent analysis of the conifold relates weak coupling closed Gromov-Witten theory to relative curve counting in a strong coupling regime \cite{AlimTischler2026}. In the scalar genus zero theory of \cite{Li2026Scalar}, a common central difference forcing was shown to reconstruct the local genus zero GV spectrum on an exceptional ray.

There is also a substantial, but logically distinct, difference equation tradition in quantum $K$-theory. Finite-difference Toda systems already appear in the work of Givental-Lee \cite{GiventalLee2003}; $q$-difference modules were later used for reconstruction in quantum $K$-theory by Iritani-Milanov-Tonita \cite{IritaniMilanovTonita2015}, and explicit $q$-difference equations were developed by Ruan-Wen \cite{RuanWen2022}. For Calabi-Yau threefolds, the quintic has been studied from the quantum $K$-theoretic $q$-difference viewpoint by Garoufalidis-Scheidegger and Wen \cite{GaroufalidisScheidegger2022,Wen2022}. These equations act on $K$-theoretic $I$- or $J$-functions by multiplicative Novikov shifts; the additive Kahler shifts considered here instead arise directly from the GV-resummed cohomological potential.

The purpose of the present paper is to retain the entire BPS genus expansion and ask what information is encoded by the common central difference forcing. A first indication is the sectorwise identity
\[
\Delta_d\Phi_{g,d}=-\Phi_{g+1,d},
\]
which shows that the degree adapted central difference raises the BPS genus by one. The main point is that after replacing the family of degree adapted operators by one common central difference, the resulting all genus forcing still admits an exact inversion.

\medskip
There are two triangular operations hidden in the Gopakumar-Vafa expansion \eqref{eq:GV-intro}. The first is arithmetic: a primitive degree $d$ contributes to every multiple $dk$, so the coefficient of $Q^N$ is a divisor sum over $d\mid N$. The second is genus theoretic: the factor
\[
\left(2\sin\frac{k\lambda}{2}\right)^{2g-2}
\]
causes a fixed BPS genus to contribute to many powers of the string coupling variable. As a consequence, neither a coefficient of $Q^N$ nor a Taylor coefficient in $\lambda$ is, by itself, a BPS invariant. Recovering the underlying integers requires undoing both kinds of mixing.

The common central difference provides a useful observable because it acts diagonally on the Novikov modes. On $Q^N$ it multiplies by
\[
-4\sin^2\frac{N\lambda}{2},
\]
so it removes the universal genus zero sine square pole without destroying the higher genus dependence. After introducing
\[
y=\left(2\sin\frac{\lambda}{2}\right)^2,
\]
this eigenvalue becomes an integral polynomial $-P_N(y)$. The resulting forcing coefficients are therefore polynomials rather than Laurent series with a genus zero pole. More importantly, the multiplication law $P_a\circ P_b=P_{ab}$ remembers the divisor arithmetic of multiple covers. This is the mechanism which turns the curve counting problem into an invertible ghost type transform.

From this point of view, the main results have a simple interpretation. The exact reconstruction theorem says that the common difference forcing loses no all genus BPS information. The genus filtration theorem says more: the order of the quantum jet measures precisely how far one has reconstructed in BPS genus. Finally, the integral image theorem asks an arithmetic question which is independent of any chosen geometric construction of the BPS invariants: which integral polynomial forcing sequences can possibly arise from integral BPS spectra? The answer separates into a familiar Frobenius/ghost condition and a second divisibility condition which records the zeros of the sine kernel.

This perspective is useful even before one turns to examples. In genus zero, setting $y=0$ collapses the theory to an ordinary weighted divisor convolution. At higher genus, the nonconstant $y$-dependence is essential rather than decorative. The transition from genus zero to all genera is therefore not obtained merely by keeping more terms in the same formula; it changes the information content of the forcing itself.

\medskip
\noindent\textbf{Relation to earlier all genus and arithmetic structures.}
The all genus multiple cover transform itself, the cover polynomials used below, and the use of M\"obius inversion in local BPS integrality go back to Bryan-Pandharipande \cite{BryanPandharipande2001}; their later TQFT solution gives a complete and effective description of the local Gromov-Witten theory of curves \cite{BryanPandharipande2008}. The present paper does not propose a new all genus GV transform. Instead, it studies the image of that transform after applying one common central difference operator. This produces a new forcing object whose $y$-adic jet, integral image, and local curve geometry can be described intrinsically.

There is also a classical arithmetic background behind the terminology used below. The relation between Witt vectors, necklace algebras, Burnside constructions and $\lambda$-ring/Adams operations has a long history \cite{MetropolisRota1983,DressSiebeneicher1988,Wilkerson1982,Elliott2006,Hazewinkel2009,Oh2006}. The Chebyshev composition operators behave as Adams/Frobenius operations, while the undressed transform has the formal shape of a Witt or necklace ghost map. The new feature here is the degree dependent Fej\'er dressing forced by the Gromov-Witten sine square denominator. This extra factor produces both the stronger integrality condition and the torsion angle vanishing of Section~\ref{sec:integrality}.

Set
\begin{equation}
y=\left(2\sin\frac{\lambda}{2}\right)^2.
\label{eq:yintro}
\end{equation}
Following Bryan-Pandharipande \cite{BryanPandharipande2001}, let $P_m(y)\in\ZZ[y]$ be defined by
\begin{equation}
P_m(4\sin^2 t)=4\sin^2(mt).
\label{eq:Pintro}
\end{equation}
Equivalently,
\[
P_m(y)=2-2T_m\left(1-\frac y2\right),
\]
where $T_m$ is the Chebyshev polynomial. These polynomials satisfy
\[
P_a\circ P_b=P_{ab}.
\]
Define the Fej\'er factor
\begin{equation}
\mathcal F_m(y)=\frac{P_m(y)}{y}
\label{eq:Fintro}
\end{equation}
and the composition operator
\[
\Psi_mf(y)=f(P_m(y)).
\]
Then $\Psi_a\Psi_b=\Psi_{ab}$.

For each degree $d$, let
\begin{equation}
B_d(y)=\sum_{g\ge0}n_{g,d}y^g
\label{eq:Bintro}
\end{equation}
be the BPS genus polynomial. Put $Q=e^{2\pi it}$, $h=\lambda/(2\pi)$, and define the common central difference
\[
\Delta=T_h-2+T_{-h}.
\]
Writing
\[
R(Q,y)=\Delta F^{\mathrm{BPS}}(Q,y)=\sum_{N\ge1}R_N(y)Q^N,
\qquad
H_N(y)=-NR_N(y),
\]
we obtain the basic transform.

\begin{theorem}[Chebyshev-Fej\'er ghost formula]
\label{thm:intro-ghost}
For every $N\ge1$,
\begin{equation}
\boxed{
H_N
=
\sum_{d\mid N}d\,\Psi_{N/d}\bigl(\mathcal F_dB_d\bigr).
}
\label{eq:intro-ghost}
\end{equation}
\end{theorem}

Define
\[
\mathfrak M_N(H)=\sum_{r\mid N}\mu(r)\Psi_r(H_{N/r}).
\]
The composition law makes ordinary divisor Mobius inversion compatible with the cover polynomial operators.

\begin{theorem}[Exact all genus reconstruction]
\label{thm:intro-recon}
For every $N\ge1$,
\begin{equation}
\boxed{
N\mathcal F_N(y)B_N(y)
=
\sum_{r\mid N}\mu(r)H_{N/r}(P_r(y)).
}
\label{eq:intro-recon}
\end{equation}
Consequently the complete central difference forcing and the complete raywise all genus BPS spectrum mutually determine one another.
\end{theorem}

The reconstruction is filtered by the string coupling variable. Since
\[
P_m(y)=m^2y+O(y^2),
\qquad
\mathcal F_m(y)=m^2+O(y),
\]
the $y$-adic filtration is preserved. After subtracting the contributions of genera $<g$, let $\rho_{g,N}$ denote the remaining coefficient of $y^g$ in $R_N$.

\begin{theorem}[BPS genus filtration]
\label{thm:intro-filtration}
For every $g\ge0$,
\begin{equation}
\boxed{
\rho_{g,N}
=
-N^{2g-1}\sum_{d\mid N}d^{3-2g}n_{g,d},
}
\label{eq:intro-graded}
\end{equation}
and
\begin{equation}
\boxed{
n_{g,N}
=
-\frac1{N^2}\sum_{r\mid N}\mu(r)r^{2g-1}\rho_{g,N/r}.
}
\label{eq:intro-graded-inv}
\end{equation}
Equivalently, for every $G\ge0$,
\[
R(Q,y)\bmod y^{G+1}
\quad\Longleftrightarrow\quad
\{n_{g,d}:0\le g\le G,\ d\ge1\}.
\]
\end{theorem}

Thus the specialization $y=0$ records exactly the genus zero spectrum, while successive quantum coefficients record higher BPS genera through a triangular reconstruction. The scalar genus zero theorem of \cite{Li2026Scalar} is precisely the $g=0$ associated graded piece.

The exact inverse also leads to an intrinsic description of the integral image. Given $H_N\in\ZZ[y]$, define
\[
A_N=\frac1N\mathfrak M_N(H).
\]
The existence of integral $A_N$ is controlled by a Witt-Dwork type ghost condition, whereas the requirement that $A_N$ arise from a polynomial BPS genus function imposes the additional divisibility $\mathcal F_N\mid A_N$.

\begin{theorem}[Integral BPS image criterion]
\label{thm:intro-integrality}
An integral polynomial sequence $H=(H_N)$ arises from a unique integral BPS sequence $B_N\in\ZZ[y]$ if and only if, for every $N$,
\begin{equation}
\boxed{
\mathfrak M_N(H)\in N\mathcal F_N(y)\ZZ[y].
}
\label{eq:intro-integrality}
\end{equation}
Equivalently, the forcing satisfies the Dwork type congruences
\[
H_N-\Psi_p(H_{N/p})\in p^{v_p(N)}\ZZ[y]
\]
for every prime $p\mid N$, and the resulting integral ghost coordinate $A_N$ is divisible by $\mathcal F_N$.
\end{theorem}

Schematically,
\[
\boxed{
\text{integral BPS forcing}
=
\text{Witt-Dwork integrality}
+
\text{Fej\'er divisibility}.
}
\]
For $N=p^a$, the criterion becomes
\[
H_{p^a}(y)-H_{p^{a-1}}(P_p(y))
\in p^a\mathcal F_{p^a}(y)\ZZ[y].
\]

The final part of the paper studies local curves. Bryan-Pandharipande solved the local Gromov-Witten theory of curves by localization, degeneration and TQFT methods \cite{BryanPandharipande2008}. We use their formulas as input and ask how the resulting data appear through the forcing. For a super-rigid elliptic curve the forcing has the exact form
\[
\exp(R_E)
=
\frac{\eta(\tau)^2}{\eta(\tau+h)\eta(\tau-h)}.
\]
For a genus $g$ local curve, the exact degree two theory yields upper BPS coefficients
\[
n_{2g-1+e,2}
=
2^{2g-1-2e}\binom{g-1}{2e}.
\]
We realize these integers as top Chern numbers on natural square root stacks over $C^{(2e)}$.

The all genus cover transform, the polynomials $P_m$, and the use of M\"obius inversion in local BPS integrality originate in \cite{BryanPandharipande2001}. The ghost/Dwork terminology reflects standard Witt vector and necklace ring structures \cite{Hazewinkel2009,Oh2006}. Our contribution is the central difference realization of a Fej\'er-dressed ghost transform, its exact genus filtration and integral image, and the local curve Chern number interpretation of the upper degree two BPS layers.

\medskip
\noindent\textbf{Why the forcing viewpoint is different.}
The reconstruction in this paper should not be read as a new proof of the classical all genus GV/BPS transform. Bryan-Pandharipande invert the multiple cover expansion in order to extract BPS integers from Gromov-Witten data. Here the input is different: we first apply a single finite-difference operator to the entire raywise potential and then ask whether the resulting forcing alone remembers the BPS spectrum. The answer is affirmative, but the inverse now carries a nontrivial Fej\'er factor. This factor is invisible in the abstract divisor M\"obius inversion and becomes decisive over $\mathbb Z[y]$: it is responsible for the difference between ordinary ghost integrality and genuine BPS integrality.

There is a second distinction at the geometric end of the paper. The square root Chern numbers are not asserted to compute raw stable map components. They occur after the BPS transform, and their equality with the upper degree two BPS integers is proved only in the balanced local curve specialization for which the required BPS polynomial is supplied by the exact Bryan-Pandharipande theory. Separating these two levels avoids conflating the virtual geometry of stable maps with the integral BPS coordinates recovered from the forcing.

\medskip
\noindent\textbf{Proof strategy.}
The argument has three layers. First, on the coefficient of $Q^N$ the common central difference contributes the cover polynomial $-P_N(y)$. Combining this with the divisor decomposition $N=dk$ converts the GV multiple cover expansion into the Chebyshev--Fej\'er ghost formula. The semigroup identity $P_a\circ P_b=P_{ab}$ then makes ordinary divisor M\"obius inversion operator-valued and yields an exact inverse. Second, the leading expansions $P_m(y)=m^2y+O(y^2)$ and $\mathcal F_m(y)=m^2+O(y)$ show that the transform preserves the $y$-adic filtration; its associated graded becomes an elementary weighted divisor convolution, which gives the uniform genus reconstruction formula. Third, integrality is separated into the classical ghost obstruction and the genuinely geometric Fej\'er obstruction. The former is equivalent to prime-power Frobenius congruences, while the latter asks the reconstructed dressed coordinate to vanish at the nontrivial torsion angles of the sine kernel.

The local curve applications use no new solution of local Gromov-Witten theory. We insert the exact Bryan-Pandharipande degree one and degree two formulas into the reconstruction. The resulting upper degree two BPS integers are then compared with a separate Grothendieck-Riemann-Roch calculation on a stack of square roots over $C^{(2e)}$. The local Calabi-Yau identity $\operatorname{ch}(\mathcal N)\operatorname{td}(C)=2$ is responsible for the unexpectedly simple Chern character.

\medskip
\noindent\textbf{A degree two model calculation.}
The basic mechanism can already be seen without any general notation. Put
\[
A_d(y)=\mathcal F_d(y)B_d(y).
\]
In degrees one and two, the ghost transform reads
\[
H_1=A_1,
\qquad
H_2=\Psi_2(A_1)+2A_2.
\]
Thus
\[
A_1=H_1,
\qquad
A_2=\frac{H_2-\Psi_2(H_1)}{2},
\]
and, since
\[
\mathcal F_2(y)=4-y,
\]
one obtains the explicit BPS reconstruction
\begin{equation}
\boxed{
B_2(y)
=
\frac{H_2(y)-H_1(P_2(y))}{2(4-y)}.
}
\label{eq:intro-degree-two-model}
\end{equation}
This single formula contains the main features of the paper. First, it exhibits the divisor subtraction which removes the degree one double cover contribution. Second, integrality of the dressed coordinate requires the numerator to be even; this is the first Dwork congruence. Third, genuine BPS integrality requires the stronger divisibility by $4-y$, the degree two Fej\'er factor. Finally, because
\[
P_2(y)=4y-y^2,
\]
the coefficient of a fixed power of $y$ receives contributions only from that BPS genus and lower genera. The reconstruction is therefore triangular in genus.

The model also explains why the all genus problem differs from the scalar genus-zero one. At $y=0$, equation \eqref{eq:intro-degree-two-model} reduces to the ordinary arithmetic reconstruction of the degree two genus zero invariant. Away from $y=0$, the same formula retains the full genus polynomial $B_2(y)$. In higher degree the expressions involve all divisors of $N$, but the composition law $P_a\circ P_b=P_{ab}$ makes the general inverse no more mysterious than this degree two calculation: M\"obius inversion performs simultaneously all of these primitive multiple cover subtractions.

The paper is organized as follows. Section~\ref{sec:calculus} develops the Chebyshev--Fej\'er central difference calculus. Section~\ref{sec:reconstruction} proves exact reconstruction and the BPS genus filtration. Section~\ref{sec:integrality} characterizes integral BPS forcing. Section~\ref{sec:localcurves} treats local curves, the elliptic $\eta$-function forcing, and the square root Chern number realization.

\section{Chebyshev--Fej\'er central difference calculus}
\label{sec:calculus}

For $g\ge0$ and $d\ge1$, define
\begin{equation}
\Phi_{g,d}(\lambda,Q)
=
\sum_{k\ge1}\frac{Q^{dk}}{k}
\left(2\sin\frac{k\lambda}{2}\right)^{2g-2}.
\label{eq:Phi}
\end{equation}
Then
\[
F^{\mathrm{BPS}}=\sum_{d\ge1}\sum_{g\ge0}n_{g,d}\Phi_{g,d}.
\]

For $a\in\QQ$, let $T_af(t)=f(t+a)$. Since $Q=e^{2\pi it}$,
\[
T_aQ^N=e^{2\pi iNa}Q^N.
\]
For each $d$, set
\[
\Delta_d=T_{h/d}-2+T_{-h/d}.
\]
On $Q^{dk}$,
\[
\Delta_dQ^{dk}=-4\sin^2\frac{k\lambda}{2}Q^{dk}.
\]

\begin{lemma}[Genus raising]
\label{lem:genus-raising}
For every $g\ge0$ and $d\ge1$,
\begin{equation}
\boxed{\Delta_d\Phi_{g,d}=-\Phi_{g+1,d}.}
\label{eq:genus-raising}
\end{equation}
\end{lemma}
\begin{proof}
Apply the preceding multiplier termwise in \eqref{eq:Phi}.
\end{proof}

For $g=0$ this gives
\[
\Delta_d\Phi_{0,d}=-\Phi_{1,d}=-\Li_1(Q^d).
\]
Thus the genus zero equation is the first step in a discrete genus ladder.

We now recall the cover polynomial calculus. The first polynomials are
\[
P_1(y)=y,
\qquad
P_2(y)=4y-y^2,
\qquad
P_3(y)=9y-6y^2+y^3,
\]
\[
P_4(y)=16y-20y^2+8y^3-y^4.
\]

\begin{lemma}[Cover polynomial calculus]
\label{lem:Pcalc}
For all positive integers $a,b,m$:
\begin{enumerate}[label=(\roman*)]
\item $P_m(y)\in y\ZZ[y]$;
\item $P_a(P_b(y))=P_{ab}(y)$;
\item $P_m(y)=m^2y+O(y^2)$;
\item $\mathcal F_m(y)=P_m(y)/y\in\ZZ[y]$ and $\mathcal F_m(0)=m^2$.
\end{enumerate}
\end{lemma}
\begin{proof}
The integrality is the standard Bryan-Pandharipande cover polynomial property \cite{BryanPandharipande2001}. The composition law follows from $T_a(T_b(x))=T_{ab}(x)$. The linear term follows by expanding $4\sin^2(mt)$ in $4\sin^2t$.
\end{proof}

If $y=4\sin^2t$, then
\begin{equation}
\mathcal F_m(y)=\left(\frac{\sin mt}{\sin t}\right)^2.
\label{eq:FejerTrig}
\end{equation}
Define
\[
\Psi_mf(y)=f(P_m(y)).
\]
Then $\Psi_a\Psi_b=\Psi_{ab}$ and
\begin{equation}
\Psi_k(\mathcal F_d)=\frac{P_{dk}(y)}{P_k(y)}.
\label{eq:PsiF}
\end{equation}

The all genus potential admits a compact polynomial form.

\begin{lemma}
\label{lem:potentialB}
One has
\begin{equation}
\boxed{
F^{\mathrm{BPS}}(Q,y)
=
\sum_{d\ge1}\sum_{k\ge1}\frac{Q^{dk}}{k}
\frac{B_d(P_k(y))}{P_k(y)}.
}
\label{eq:potentialB}
\end{equation}
\end{lemma}
\begin{proof}
For fixed $(d,k)$,
\[
\sum_{g\ge0}n_{g,d}P_k(y)^{g-1}=\frac{B_d(P_k(y))}{P_k(y)}.
\]
\end{proof}

Now define the common central difference
\[
\Delta=T_h-2+T_{-h}.
\]
On $Q^N$,
\begin{equation}
\Delta Q^N=-P_N(y)Q^N.
\label{eq:commonmode}
\end{equation}
Write
\[
R(Q,y)=\Delta F^{\mathrm{BPS}}(Q,y)=\sum_{N\ge1}R_N(y)Q^N,
\qquad
H_N=-NR_N.
\]

\begin{proposition}[Common all genus forcing]
\label{prop:common-forcing}
For every $N\ge1$,
\begin{equation}
\boxed{
R_N(y)
=
-\frac{P_N(y)}{N}
\sum_{d\mid N}d\,
\frac{B_d(P_{N/d}(y))}{P_{N/d}(y)}.
}
\label{eq:commonforcing}
\end{equation}
\end{proposition}
\begin{proof}
The coefficient of $Q^N$ in \eqref{eq:potentialB} receives a contribution from $(d,k)$ precisely when $dk=N$, with coefficient
\[
\frac dN\frac{B_d(P_{N/d})}{P_{N/d}}.
\]
Equation \eqref{eq:commonmode} multiplies the whole $Q^N$ coefficient by $-P_N$.
\end{proof}

\begin{corollary}[Polynomial regularization]
\label{cor:regularization}
If $B_d\in A[y]$ for a commutative ring $A$, then $H_N\in A[y]$ for all $N$.
\end{corollary}
\begin{proof}
If $k=N/d$, then $P_N=P_d(P_k)$ and hence
\[
\frac{P_N}{P_k}=\mathcal F_d(P_k)\in\ZZ[y].
\]
\end{proof}

\begin{proof}[Proof of Theorem~\ref{thm:intro-ghost}]
From Proposition~\ref{prop:common-forcing},
\[
H_N=P_N\sum_{d\mid N}d\frac{B_d(P_{N/d})}{P_{N/d}}.
\]
Using \eqref{eq:PsiF} with $k=N/d$ gives
\[
H_N=\sum_{d\mid N}d\,\Psi_{N/d}(\mathcal F_dB_d).
\]
\end{proof}

Define the dressed BPS coordinate
\begin{equation}
A_d(y)=\mathcal F_d(y)B_d(y).
\label{eq:Ad}
\end{equation}
Then
\begin{equation}
\boxed{H_N=\sum_{d\mid N}d\,\Psi_{N/d}(A_d).}
\label{eq:undressedghost}
\end{equation}
The second arrow has the formal shape of a Witt/necklace ghost transform, whereas the first arrow is the Fej\'er dressing specific to the central difference problem.

At $y=0$,
\[
H_N(0)=\sum_{d\mid N}d^3n_{0,d},
\qquad
R_N(0)=-\frac1N\sum_{d\mid N}d^3n_{0,d}.
\]
This recovers the genus zero divisor convolution of \cite{Li2026Scalar}.

\section{Reconstruction and the BPS genus filtration}
\label{sec:reconstruction}

Define
\begin{equation}
\mathfrak M_N(H)=\sum_{r\mid N}\mu(r)\Psi_r(H_{N/r}).
\label{eq:MobiusOp}
\end{equation}

\begin{theorem}[Operator-valued Mobius inversion]
\label{thm:operatorMobius}
For every $N\ge1$,
\begin{equation}
\boxed{\mathfrak M_N(H)=NA_N.}
\label{eq:operatorMobius}
\end{equation}
Equivalently,
\begin{equation}
\boxed{
B_N(y)=\frac{1}{N\mathcal F_N(y)}
\sum_{r\mid N}\mu(r)H_{N/r}(P_r(y)).
}
\label{eq:exactBN}
\end{equation}
\end{theorem}
\begin{proof}
Substitute \eqref{eq:undressedghost} into \eqref{eq:MobiusOp} and use $\Psi_a\Psi_b=\Psi_{ab}$:
\[
\mathfrak M_N(H)
=
\sum_{d\mid N}d\,\Psi_{N/d}(A_d)
\sum_{r\mid N/d}\mu(r).
\]
Only $d=N$ survives.
\end{proof}

Work over $\QQ[[y]]$ with the filtration $F^g=y^g\QQ[[y]]$. Since $P_m(y)=m^2y+O(y^2)$, the operator $\Psi_m$ acts on $\operatorname{gr}^g_F$ by multiplication by $m^{2g}$, while multiplication by $\mathcal F_d$ acts by $d^2$.

For fixed $g$, let $H_N^{<g}$ denote the forcing reconstructed from the BPS truncations
\[
B_d^{<g}(y)=\sum_{j=0}^{g-1}n_{j,d}y^j,
\]
and define
\[
\eta_{g,N}=[y^g](H_N-H_N^{<g}).
\]

\begin{proposition}[Associated graded ghost transform]
\label{prop:gradedghost}
For every $g\ge0$,
\begin{equation}
\boxed{
\eta_{g,N}=N^{2g}\sum_{d\mid N}d^{3-2g}n_{g,d}.
}
\label{eq:etaGraded}
\end{equation}
\end{proposition}
\begin{proof}
The pure genus-$g$ term $n_{g,d}y^g$ acquires the factor $d^2$ from $\mathcal F_d$, the factor $(N/d)^{2g}$ from $\Psi_{N/d}$, and the outer ghost factor $d$.
\end{proof}

\begin{theorem}[Associated graded reconstruction]
\label{thm:gradedrecon}
For every $g\ge0$,
\begin{equation}
\boxed{
n_{g,N}=\frac1{N^3}\sum_{r\mid N}\mu(r)r^{2g}\eta_{g,N/r}.
}
\label{eq:etaRecon}
\end{equation}
\end{theorem}
\begin{proof}
Divide \eqref{eq:etaGraded} by $N^{2g}$ and apply ordinary Mobius inversion.
\end{proof}

In terms of the original forcing, define
\[
\rho_{g,N}=[y^g](R_N-R_N^{<g}).
\]
Since $\eta_{g,N}=-N\rho_{g,N}$, we obtain
\begin{equation}
\boxed{
\rho_{g,N}
=
-N^{2g-1}\sum_{d\mid N}d^{3-2g}n_{g,d},
}
\label{eq:rhoGraded}
\end{equation}
\begin{equation}
\boxed{
n_{g,N}
=
-\frac1{N^2}\sum_{r\mid N}\mu(r)r^{2g-1}\rho_{g,N/r}.
}
\label{eq:rhoRecon}
\end{equation}
The outer factor $N^{-2}$ is independent of genus; only the Mobius weight changes.

\begin{theorem}[Quantum jet reconstruction]
\label{thm:quantumjet}
For every $G\ge0$, the following data are equivalent:
\begin{enumerate}[label=(\roman*)]
\item $\{H_N(y)\bmod y^{G+1}\}_{N\ge1}$;
\item $\{B_N(y)\bmod y^{G+1}\}_{N\ge1}$;
\item $\{n_{g,N}:0\le g\le G,\ N\ge1\}$.
\end{enumerate}
\end{theorem}
\begin{proof}
The forward transform preserves the filtration. Conversely, Theorem~\ref{thm:operatorMobius} reconstructs $\mathcal F_NB_N$, and $\mathcal F_N(0)=N^2\ne0$, so $\mathcal F_N$ is a unit in $\QQ[[y]]$.
\end{proof}

For reference, the cover polynomial expansion is
\begin{equation}
P_m(y)=m^2y-\frac{m^2(m^2-1)}{12}y^2
+\frac{m^2(m^2-1)(m^2-4)}{360}y^3+O(y^4).
\label{eq:Pexpansion}
\end{equation}
Thus
\[
\mathcal F_m(y)=m^2-\frac{m^2(m^2-1)}{12}y
+\frac{m^2(m^2-1)(m^2-4)}{360}y^2+O(y^3).
\]

At genus zero,
\[
n_{0,N}=\frac1{N^3}\sum_{r\mid N}\mu(r)H_{N/r}(0)
=-\frac1{N^2}\sum_{r\mid N}\frac{\mu(r)}rR_{N/r}(0).
\]
At genus one, if $h_{1,N}=[y]H_N$, define
\[
\eta_{1,N}
=
h_{1,N}
+\frac{N^2}{12}\sum_{d\mid N}d(d^2-1)n_{0,d}.
\]
Then
\[
n_{1,N}=\frac1{N^3}\sum_{r\mid N}\mu(r)r^2\eta_{1,N/r}.
\]
At genus two, if $h_{2,N}=[y^2]H_N$, define
\begin{align*}
\eta_{2,N}
={}&h_{2,N}
+\frac{N^2(N^2-1)}{12}\sum_{d\mid N}d\,n_{1,d}\\
&-\frac{N^2}{720}\sum_{d\mid N}
\frac{d^2-1}{d}
\left[N^2(2d^2-3)-5d^2\right]n_{0,d}.
\end{align*}
Then
\[
n_{2,N}=\frac1{N^3}\sum_{r\mid N}\mu(r)r^4\eta_{2,N/r}.
\]
These explicit formulas illustrate the triangular contamination by lower genera; the associated graded formula \eqref{eq:etaRecon} is uniform in $g$.

\section{Arithmetic characterization of integral BPS forcing}
\label{sec:integrality}

We now characterize the integral image of the transform. First forget the Fej\'er dressing and consider integral polynomials $A_d\in\ZZ[y]$ with
\begin{equation}
H_N=\sum_{d\mid N}d\,\Psi_{N/d}(A_d).
\label{eq:undressedSec4}
\end{equation}
Then
\[
\mathfrak M_N(H)=NA_N.
\]

\begin{lemma}[Frobenius congruence]
\label{lem:frob}
For every prime $p$,
\begin{equation}
\boxed{P_p(y)\equiv y^p\pmod p.}
\label{eq:PpFrob}
\end{equation}
Consequently, for $f\in\ZZ[y]$,
\[
\Psi_p(f)\equiv f^p\pmod p.
\]
\end{lemma}
\begin{proof}
Introduce $z$ by $y=2-z-z^{-1}$. Then $P_p(y)=2-z^p-z^{-p}$, while
\[
(2-z-z^{-1})^p\equiv2-z^p-z^{-p}\pmod p.
\]
\end{proof}

\begin{theorem}[Chebyshev-Dwork ghost criterion]
\label{thm:Dwork}
For $H_N\in\ZZ[y]$, the following are equivalent:
\begin{enumerate}[label=(\roman*)]
\item there exists a unique sequence $A_N\in\ZZ[y]$ satisfying \eqref{eq:undressedSec4};
\item for every prime $p\mid N$,
\begin{equation}
\boxed{
H_N-\Psi_p(H_{N/p})\in p^{v_p(N)}\ZZ[y].
}
\label{eq:Dwork}
\end{equation}
\end{enumerate}
\end{theorem}
\begin{proof}
If $N=p^am$ with $p\nmid m$, subtracting $\Psi_p(H_{N/p})$ from $H_N$ cancels the terms with $v_p(d)<a$, leaving
\[
p^a\sum_{e\mid m}e\,\Psi_{m/e}(A_{p^ae}).
\]
Conversely, let $M_N=\mathfrak M_N(H)$. Split the Mobius divisors as $r=s$ and $r=ps$, $s\mid m$. Then
\[
M_N=\sum_{s\mid m}\mu(s)\Psi_s\left[H_{p^am/s}-\Psi_p(H_{p^{a-1}m/s})\right].
\]
Hence $p^a\mid M_N$ for every $p^a\Vert N$, so $N\mid M_N$. Set $A_N=M_N/N$.
\end{proof}

The BPS image imposes the additional condition $A_N=\mathcal F_NB_N$.

\begin{theorem}[Integral BPS image criterion]
\label{thm:BPSimage}
Let $H_N\in\ZZ[y]$. The following are equivalent:
\begin{enumerate}[label=(\roman*)]
\item there exists a unique $B_N\in\ZZ[y]$ with
\[
H_N=\sum_{d\mid N}d\,\Psi_{N/d}(\mathcal F_dB_d);
\]
\item for every $N$,
\begin{equation}
\boxed{
\mathfrak M_N(H)\in N\mathcal F_N(y)\ZZ[y];
}
\label{eq:BPSimage}
\end{equation}
\item the Dwork congruences \eqref{eq:Dwork} hold and $\mathcal F_N\mid A_N$, where $A_N=\mathfrak M_N(H)/N$.
\end{enumerate}
\end{theorem}
\begin{proof}
This is Theorem~\ref{thm:Dwork} combined with Theorem~\ref{thm:operatorMobius}.
\end{proof}

Thus
\[
\boxed{
\text{integral BPS forcing}
=
\text{Witt-Dwork integrality}
+
\text{Fej\'er divisibility}.
}
\]
The second condition is genuinely additional. Indeed, taking $A_d=1$ gives $H_N=\sigma_1(N)$, which satisfies all Dwork congruences, but for $N>1$ the polynomial $A_N=1$ is not divisible by $\mathcal F_N$.

For a prime power $N=p^a$, the Mobius sum has only two terms:
\begin{equation}
\boxed{
B_{p^a}(y)
=
\frac{H_{p^a}(y)-H_{p^{a-1}}(P_p(y))}
{p^a\mathcal F_{p^a}(y)}.
}
\label{eq:primePower}
\end{equation}
Thus integral BPS forcing satisfies
\[
H_{p^a}(y)-H_{p^{a-1}}(P_p(y))
\in p^a\mathcal F_{p^a}(y)\ZZ[y].
\]

The Fej\'er factor also has a torsion angle interpretation. Set
\[
y_{j,N}=4\sin^2\frac{\pi j}{N}.
\]

\begin{proposition}[Torsion angle factorization]
\label{prop:torsion}
If $N$ is odd,
\[
\mathcal F_N(y)
=
\prod_{j=1}^{(N-1)/2}(y-y_{j,N})^2.
\]
If $N$ is even,
\[
\mathcal F_N(y)
=
(4-y)\prod_{j=1}^{N/2-1}(y-y_{j,N})^2.
\]
\end{proposition}
\begin{proof}
Use \eqref{eq:FejerTrig}. The nontrivial zeros occur at $t=\pi j/N$. Paired values $j$ and $N-j$ give the same $y$ and yield double roots. For even $N$, the unpaired value $j=N/2$ gives the simple root $y=4$.
\end{proof}

Hence the BPS image criterion may also be read as arithmetic Frobenius congruences plus prescribed torsion angle vanishing.

\section{Local curves, modular forcing and square root BPS geometry}
\label{sec:localcurves}

We now apply the general formalism to local curves. Hodge integral calculations already encode universal Calabi-Yau multiple cover contributions \cite{FaberPandharipande2000}, and the local Gromov-Witten theory of curves was subsequently solved by Bryan-Pandharipande \cite{BryanPandharipande2008}; their formulas are used here as input.

\subsection{The super rigid elliptic curve}

For a super-rigid elliptic curve, the local BPS data are
\[
n_{r,d}=\begin{cases}1,&r=1,\\0,&r\ne1,\end{cases}
\qquad d\ge1,
\]
so $B_d(y)=y$. The potential is
\[
F_E(Q)=\sum_{d\ge1}\Li_1(Q^d)
=-\log\prod_{d\ge1}(1-Q^d).
\]
With $Q=e^{2\pi i\tau}$,
\[
F_E(\tau)=-\log\eta(\tau)+\frac{\pi i\tau}{12}.
\]
Therefore
\begin{equation}
\boxed{
R_E(\tau,\lambda)
=-\log\eta(\tau+h)+2\log\eta(\tau)-\log\eta(\tau-h),
\qquad h=\frac{\lambda}{2\pi}.
}
\label{eq:eta}
\end{equation}
Equivalently,
\[
\exp(R_E)=\frac{\eta(\tau)^2}{\eta(\tau+h)\eta(\tau-h)}.
\]
Coefficientwise,
\[
H_{E,N}(y)=\sigma_1(N)P_N(y),
\]
and Theorem~\ref{thm:operatorMobius} recovers $B_N=y$ using $\mu*\sigma_1=\operatorname{id}$.

\subsection{The exact degree two local curve}

Assume now $g\ge2$ and consider the balanced local Calabi-Yau specialization. Bryan-Pandharipande give the degree two disconnected partition coefficient \cite[Section~8]{BryanPandharipande2008}:
\begin{equation}
\boxed{
Z_2^{(g)}(y)
=
y^{2g-2}
\left[(4-2\sqrt y)^{g-1}+(4+2\sqrt y)^{g-1}\right].
}
\label{eq:Z2}
\end{equation}
The degree one coefficient is $Z_1^{(g)}=y^{g-1}$, so the connected degree two free energy is
\begin{equation}
F_2^{(g)}(y)
=
y^{2g-2}
\left[(4-2\sqrt y)^{g-1}+(4+2\sqrt y)^{g-1}-\frac12\right].
\label{eq:F2}
\end{equation}
Since $B_1(y)=y^g$ and $P_2(y)=y(4-y)$, the degree two GV relation gives
\[
F_2^{(g)}(y)
=
\frac{B_2^{(g)}(y)}y
+\frac12\frac{B_1(P_2(y))}{P_2(y)}.
\]
Hence
\begin{theorem}[Closed degree two BPS polynomial]
\label{thm:B2closed}
\begin{equation}
\boxed{
B_2^{(g)}(y)
=
y^g\left\{
y^{g-1}\left[(4-2\sqrt y)^{g-1}+(4+2\sqrt y)^{g-1}-\frac12\right]
-\frac12(4-y)^{g-1}
\right\}.
}
\label{eq:B2closed}
\end{equation}
The right hand side belongs to $\ZZ[y]$.
\end{theorem}
\begin{proof}
Multiply the degree two GV relation by $y$ and substitute \eqref{eq:F2}. The symmetric pair in $\sqrt y$ contains only integral powers of $y$.
\end{proof}

Expanding the symmetric pair yields the upper ladder.

\begin{corollary}[Upper degree two BPS ladder]
\label{cor:upperB2}
For every $e\ge1$,
\begin{equation}
\boxed{
n_{2g-1+e,2}
=
2^{2g-1-2e}\binom{g-1}{2e}.
}
\label{eq:upperB2}
\end{equation}
In particular the invariant vanishes for $2e>g-1$.
\end{corollary}

The critical unbranched layer is exceptional:
\begin{equation}
\boxed{
n_{2g-1,2}
=
2^{2g-1}-\frac{1+(-1)^{g-1}}2.
}
\label{eq:criticalB2}
\end{equation}
This equals the weighted count of connected unramified double covers, $(2^{2g}-1)/2$, minus the universal degree one double multiple cover contribution $(-1)^{g-1}/2$.

\subsection{Square root stacks and the anti-invariant local Calabi-Yau obstruction}

Moduli of roots, twisted covers and root constructions have a substantial history \cite{AbramovichCortiVistoli2003,Cadman2007}. On the Riemann--Roch side, Chiodo's formula describes Chern characters of derived pushforwards of universal $r$-th roots and has later applications to Hurwitz theory and topological recursion \cite{Chiodo2008,LewanskiEtAl2017}; see also \cite{ChiodoHolmes2023} for a universal Picard formulation. Branch morphisms from stable map spaces to symmetric products provide related covering theoretic geometry \cite{FantechiPandharipande2002}. Our construction below is different in base and in enumerative interpretation: the curve $C$ is fixed, the root varies with an effective divisor $D\in C^{(2e)}$, and the derived pushforward is twisted by the local Calabi-Yau normal bundle. We use this fixed curve square root geometry only to realize a specific family of degree two BPS integers.

Recent refined DT/PT work gives further all genus refinements of local curve partition functions \cite{Monavari2025}; we do not use those results here, but they provide a complementary setting in which refinements of the local curve BPS package can be studied.

Let $\mathcal N\to C$ be a rank two bundle satisfying
\begin{equation}
\det\mathcal N\simeq K_C.
\label{eq:localCY}
\end{equation}
For $e\ge1$, let
\[
a_{2e}:C^{(2e)}\longrightarrow \Pic^{2e}(C),
\qquad
D\longmapsto \mathcal O_C(D),
\]
and let
\[
[2]:\Pic^e(C)\longrightarrow \Pic^{2e}(C),
\qquad
M\longmapsto M^{\otimes 2}.
\]
The stack of square root data is
\begin{equation}
\mathscr R_{2e}(C)
=
\left\{(D,M,\phi):D\in C^{(2e)},\ M\in\Pic^e(C),\
\phi:M^{\otimes2}\xrightarrow{\sim}\mathcal O_C(D)\right\}.
\label{eq:rootstack}
\end{equation}
Equivalently, if
\[
\widetilde C^{(2e)}
:=
C^{(2e)}\times_{\Pic^{2e}(C),\,a_{2e}}\Pic^e(C)
\]
is formed using the multiplication-by-two map, then
$\mathscr R_{2e}(C)\to\widetilde C^{(2e)}$ is the $\mu_2$-gerbe of
trivializations of the universal square.  Since $[2]$ is finite etale of degree
$2^{2g}$ and every square root datum has automorphism group $\mu_2$, the
forgetful morphism
\[
\varrho:\mathscr R_{2e}(C)\longrightarrow C^{(2e)}
\]
has stack theoretic degree $2^{2g-1}$.  More precisely, in the rational Chow
group,
\begin{equation}
\boxed{
\varrho_*[\mathscr R_{2e}(C)]
=
2^{2g-1}[C^{(2e)}].
}
\label{eq:rootdegree}
\end{equation}
Consequently, for every top degree rational Chow class $\alpha$ on $C^{(2e)}$,
\begin{equation}
\int_{\mathscr R_{2e}(C)}\varrho^*\alpha
=
2^{2g-1}\int_{C^{(2e)}}\alpha.
\label{eq:rootpush}
\end{equation}
As a stack of root data, $\mathscr R_{2e}(C)$ carries a tautological line bundle
$\mathcal M$ on $C\times\mathscr R_{2e}(C)$ together with a tautological
isomorphism
\[
\mathcal M^{\otimes2}\simeq\mathcal O(\mathcal D),
\]
where $\mathcal D$ is the pullback of the universal divisor on
$C\times C^{(2e)}$.

Let $q:C\times\mathscr R_{2e}(C)\to C$ and
$p:C\times\mathscr R_{2e}(C)\to\mathscr R_{2e}(C)$ be the projections. Define
\begin{equation}
\boxed{
\mathbb V_e=-Rp_*(q^*\mathcal N\otimes\mathcal M^{-1}).
}
\label{eq:Ve}
\end{equation}
Over the reduced branch locus the root datum determines a double cover $f:\widetilde C\to C$ and
\[
f_*\mathcal O_{\widetilde C}\simeq\mathcal O_C\oplus M^{-1}.
\]
Thus $\mathbb V_e$ is the family version of the cover-dependent anti-invariant part of $R\Gamma(\widetilde C,f^*\mathcal N)$.

Let $x$ be the class of the divisor $p_0+C^{(2e-1)}$ and let $\theta$ be the pullback of the theta class under Abel-Jacobi. The universal divisor has Kunneth decomposition
\[
[\mathcal D]=2e\,\eta+\gamma+x,
\qquad
\gamma^2=-2\eta\theta.
\]

\begin{proposition}[Local Calabi-Yau GRR cancellation]
\label{prop:GRR}
The class $\mathbb V_e$ has virtual rank $2e$ and
\begin{equation}
\boxed{
\ch(\mathbb V_e)=e^{-x/2}\left(2e+\frac{\theta}{2}\right).
}
\label{eq:chVe}
\end{equation}
Consequently
\begin{equation}
\boxed{
c_t(\mathbb V_e)
=
\left(1-\frac{xt}{2}\right)^{2e}
\exp\left(\frac{\theta t/2}{1-xt/2}\right).
}
\label{eq:ctVe}
\end{equation}
\end{proposition}
\begin{proof}
Because $\det\mathcal N\simeq K_C$,
\[
\ch(\mathcal N)=2+(2g-2)\eta,
\qquad
\td(C)=1+(1-g)\eta,
\]
and hence
\begin{equation}
\ch(\mathcal N)\td(C)=2.
\label{eq:CYcancel}
\end{equation}
Since $c_1(\mathcal M)=[\mathcal D]/2$, Grothendieck--Riemann--Roch in the form used below (cf. \cite{Fulton1998}) gives
\[
\ch Rp_*(q^*\mathcal N\otimes\mathcal M^{-1})
=p_*\left[2e^{-[\mathcal D]/2}\right]
=e^{-x/2}\left(-2e-\frac{\theta}{2}\right).
\]
Applying the minus sign in \eqref{eq:Ve} gives \eqref{eq:chVe}. The total Chern class follows from the standard relation between Chern character and total Chern class.
\end{proof}

The top Chern class is
\begin{equation}
c_{2e}(\mathbb V_e)
=
\frac1{2^{2e}}\sum_{i=0}^{2e}(-1)^{2e-i}\frac{x^{2e-i}\theta^i}{i!}.
\label{eq:topVe}
\end{equation}
Macdonald's intersection calculus for symmetric products gives \cite{Macdonald1962}
\begin{equation}
\int_{C^{(n)}}x^{n-i}\theta^i=\frac{g!}{(g-i)!}.
\label{eq:MacInt}
\end{equation}
Therefore
\[
\int_{C^{(2e)}}c_{2e}(\mathbb V_e)
=
4^{-e}\binom{g-1}{2e}.
\]
The Chern character in Proposition~\ref{prop:GRR}, and hence the top Chern class above, is expressed entirely in the pullbacks of $x$ and $\theta$ from $C^{(2e)}$. Therefore the pushforward identity \eqref{eq:rootpush} applies directly and yields the square root Chern number below.

\begin{proposition}[Square root Chern number]
\label{prop:rootChern}
For every $e\ge1$,
\begin{equation}
\boxed{
\int_{\mathscr R_{2e}(C)}c_{2e}(\mathbb V_e)
=
2^{2g-1-2e}\binom{g-1}{2e}.
}
\label{eq:rootChern}
\end{equation}
This formula depends only on the local Calabi-Yau condition $\det\mathcal N\simeq K_C$.
\end{proposition}
\begin{proof}
Combine Proposition~\ref{prop:GRR}, \eqref{eq:MacInt}, and the rational Chow pushforward formula \eqref{eq:rootpush}.
\end{proof}

\begin{corollary}[Balanced degree two BPS realization]
\label{cor:rootrealization}
In the balanced local Calabi-Yau specialization used in Theorem~\ref{thm:B2closed}, for
\[
1\le e\le\left\lfloor\frac{g-1}{2}\right\rfloor,
\]
one has
\begin{equation}
\boxed{
 n_{2g-1+e,2}
=
\int_{\mathscr R_{2e}(C)}c_{2e}(\mathbb V_e).
}
\label{eq:rootrealization}
\end{equation}
\end{corollary}
\begin{proof}
Compare Proposition~\ref{prop:rootChern} with Corollary~\ref{cor:upperB2}.
\end{proof}

We emphasize that Corollary~\ref{cor:rootrealization} is a BPS level Chern number realization in the balanced local curve theory. We do not identify $\mathscr R_{2e}(C)$ with a compactification of a single stable map component whose virtual degree is the raw Gromov-Witten coefficient.

Combining Corollary~\ref{cor:rootrealization} with the forcing reconstruction in the balanced local curve theory gives, for $e\ge1$,
\begin{equation}
\boxed{
[y^{2g-1+e}]
\frac{H_2(y)-H_1(P_2(y))}{2(4-y)}
=
\int_{\mathscr R_{2e}(C)}c_{2e}(\mathbb V_e).
}
\label{eq:bridge}
\end{equation}
Thus, in the balanced local curve theory, the same integer is simultaneously a coefficient reconstructed from the central difference forcing, a BPS invariant, and a square root stack Chern number.

\subsection{A degree three diagnostic}

The square root construction is special to degree two. For a finite flat Gorenstein cover $f:D\to C$ of degree $d$, define the Tschirnhausen bundle $E$ by
\begin{equation}
0\to\mathcal O_C\to f_*\mathcal O_D\to E^\vee\to0.
\label{eq:Tsch}
\end{equation}
Then $\rk E=d-1$. For degree two this is a line bundle, while generic degree three covers have $S_3$ monodromy and are governed by rank two Tschirnhausen data rather than cubic roots; compare Miranda \cite{Miranda1985} and Casnati-Ekedahl \cite{CasnatiEkedahl1996}.

If the source genus is written as
\[
G-1=d(g-1)+e,
\]
then $\deg E=e$, and the cover-dependent local CY obstruction has the natural form
\[
-R\pi_*(E^\vee\oplus K_C\otimes E^\vee),
\]
whose virtual rank is $2e$. By Riemann-Hurwitz, $2e$ is also the branch degree. This suggests a higher degree Hurwitz/Tschirnhausen extension of the square root picture, but we do not develop it here.

As a diagnostic, the weighted connected unramified degree three cover count is
\begin{equation}
C_{3,g}
=
2\cdot6^{2g-2}+3^{2g-2}-2^{2g-1}+\frac13.
\label{eq:C3}
\end{equation}
The critical primitive degree three BPS layer removes the universal degree one $k=3$ multiple cover term and becomes
\[
n_{3g-2,3}
=
2\cdot6^{2g-2}+3^{2g-2}-2^{2g-1}.
\]
At the forcing level the subtraction is automatic:
\[
B_3(y)=\frac{H_3(y)-H_1(P_3(y))}{3\mathcal F_3(y)},
\qquad
\mathcal F_3(y)=(3-y)^2.
\]
This shows that the reconstruction is not peculiar to degree two, while the correct geometry beyond degree two is naturally higher rank.

\section{Concluding perspective}
\label{sec:conclusion}

The results above isolate a common algebraic structure behind three features of the all genus raywise Gopakumar-Vafa expansion. First, the common central difference regularizes the genus zero pole and converts the multiple cover expansion into a polynomial forcing sequence. Second, the cover polynomial composition law makes this forcing exactly invertible: after Fej\'er dressing, ordinary divisor M\"obius inversion survives as an operator-valued ghost inversion. Third, the same transform is compatible with the $y$-adic filtration, so the quantum jet has an intrinsic enumerative meaning as the BPS genus filtration.

This point of view also separates two notions of integrality which are easy to conflate at the level of the original generating function. The Witt-Dwork congruences characterize integral ghost coordinates for the Chebyshev Frobenius system. BPS integrality is strictly stronger because the dressed coordinate must also contain the full Fej\'er factor. Equivalently, besides satisfying the arithmetic Frobenius congruences, it must vanish with prescribed multiplicity at the nontrivial torsion angles of the sine kernel. Thus the central difference remembers both the divisor arithmetic of multiple covers and the analytic zero set of the universal multiple cover denominator.

The local curve calculations show that these structures are not merely formal. For an elliptic curve the forcing is an exact finite difference of $\log\eta$, while in higher genus the degree two reconstruction produces integers which admit a separate Chern number realization on square root stacks. The local Calabi-Yau identity
\[
\operatorname{ch}(\mathcal N)\operatorname{td}(C)=2
\]
is precisely what makes the corresponding Grothendieck--Riemann--Roch calculation collapse to a simple expression on the symmetric product. In this sense the algebraic simplicity of the forcing and the geometric simplicity of the Chern character reflect the same Calabi-Yau cancellation from two different sides.

Taken together, the paper identifies central difference forcing as a lossless all-genus coordinate system along an effective ray: over $\mathbb Q[[y]]$ it is equivalent to the BPS spectrum, its finite jets recover the finite genus truncations, and over $\mathbb Z[y]$ its image is cut out by explicit arithmetic and Fej\'er conditions. The local curve examples provide concrete geometric models for these abstract statements without requiring a reformulation of the underlying Gromov-Witten theory.

\medskip
\noindent\textbf{Three coordinate systems.}
It is useful to separate the three types of data which appear throughout the paper. The original Gromov-Witten expansion is naturally expressed in terms of coefficients of $Q^N$ and powers of the string coupling, but these coordinates mix primitive curve degrees with multiple covers and mix BPS genera with the Taylor expansion of the sine factors. The BPS polynomials
\[
B_N(y)=\sum_{g\ge0}n_{g,N}y^g
\]
remove this mixing and are the enumeratively natural coordinates, but they are not the quantities on which a common difference operator acts diagonally. The normalized forcing coefficients $H_N$ occupy an intermediate position: they are obtained directly from the generating function by one universal operator, yet they retain enough arithmetic structure to reconstruct the BPS polynomials exactly.

Thus one may view the paper as establishing the chain
\[
\boxed{
\text{Gromov-Witten expansion}
\ \xrightarrow{\ \Delta\ }\
\text{difference forcing}
\ \xleftrightarrow{\ \text{M\"obius--Fej\'er}\ }\
\text{BPS genus polynomials}.
}
\]
The left arrow is analytic in origin: it uses the central difference to cancel the universal sine-square pole. The right equivalence is arithmetic and algebraic: divisor M\"obius inversion removes multiple cover mixing, while Fej\'er division removes the remaining kernel imposed by the difference operator. The genus filtration is visible on the right as polynomial degree and on the middle term as the $y$-adic filtration. This explains why the forcing simultaneously interacts with curve counting geometry, Frobenius congruences, and finite-difference equations.

The same distinction clarifies the role of the local curve Chern numbers. They provide a geometric realization of selected BPS coordinates, not an alternative definition of the raw Gromov-Witten coefficients. Equation \eqref{eq:bridge} then says that these Chern numbers can nevertheless be read directly from the forcing after the universal primitive multiple cover subtraction. In this sense the difference forcing serves as a bridge between the analytic presentation of the generating function and the integral geometric quantities which the BPS transform isolates.

\section*{Acknowledgements}
The author is supported by the National Natural Science Foundation of China
(NSFC) under Grant Nos.~11501470, 11426187, and 11791240561, and is partially
supported by NSFC Grant No.~11671328, the Chengdu Science and Technology
Program under Grant No.~2025-YF09-00007-SN, and the Fundamental Research
Funds for the Central Universities under Grant Nos.~2682021ZTPY043,
2682025ZTPY001, 2682025ZTPY057, and 2682025ZTO002.

The author would like to thank Sung-Soo Kim, Kimyeong Lee, Futoshi Yagi,
Satoshi Nawata, and Rui-Dong Zhu for useful discussions and collaborations in
mathematical physics. The author is particularly grateful to Andrea Brini, Todar Milanov and Arpan Saha for
inspiring online seminars during the pandemic, and to Doan Nhat Minh and Van
Nguyen for enjoyable and fruitful collaborations. The author also thanks
Bohui Chen, An-Min Li, and Guosong Zhao for their constant supports, as well
as friends and colleagues for stimulating discussions at various conferences.
Finally, the author gratefully acknowledges the hospitality and support of the
Mainz Institute for Theoretical Physics (MITP) of the Cluster of Excellence
PRISMA+ (Project ID 390831469).

\end{document}